\documentclass[12pt]{article}

\usepackage{amsfonts}
\usepackage{amssymb}
\usepackage{amsthm}
\usepackage{amsmath}
\usepackage{a4wide}
\usepackage{mathrsfs}
\usepackage{epsfig}
\usepackage{palatino}
\usepackage{tikz,fp,ifthen}
\usepackage{float}
\usepackage{graphicx,cite}

\usepackage{graphicx}

\usepackage{color}
\definecolor{citegreen}{rgb}{0,0.6,0}
\definecolor{refred}{rgb}{0.8,0,0}
\usepackage[colorlinks, citecolor=citegreen,linkcolor=refred]{hyperref}

\usepackage{mathtools}
\mathtoolsset{showonlyrefs}

\newtheorem{thm}{Theorem}[section]
\newtheorem{lem}[thm]{Lemma}
\newtheorem{prop}[thm]{Proposition}
\newtheorem{cor}[thm]{Corollary}

\theoremstyle{definition}
\newtheorem{defn}[thm]{Definition}

\theoremstyle{remark}
\newtheorem{rem}[thm]{Remark}

\numberwithin{equation}{section}

\def\Ric{{\mathrm {Ric}}}

\def\R{\mathbb R}

\def\R{{{\mathbb R}}}

\def\NN{\mathbb N}

\newcommand{\intbar}{\etaathop{\int\etaakebox(-13.5,0){\rule[4pt]{.7em}{0.3pt}}
\kern-6pt}\nolimits}

\newcommand{\be}{\begin{equation}}
\newcommand{\ee}{\end{equation}}
\newcommand{\bea}{\begin{equation*}}
\newcommand{\eea}{\end{equation*}}

\usepackage{tikz,fp,ifthen,fullpage}
\usetikzlibrary{backgrounds}
\usetikzlibrary{decorations.pathmorphing,backgrounds,fit,calc,through}
\usetikzlibrary{arrows}
\usetikzlibrary{shapes,decorations,shadows}
\usetikzlibrary{fadings}
\usetikzlibrary{patterns}
\usetikzlibrary{mindmap}
\usetikzlibrary{decorations.text}
\usetikzlibrary{decorations.shapes}

\title{Ancient solutions of superlinear heat equations on Riemannian manifolds}

\author{Daniele Castorina \footnote{Dipartimento di Matematica, Universit\`a di Roma Tor Vergata, Via della Ricerca Scientifica, 00133
Roma, Italy, castorin@mat.uniroma2.it} \and Carlo Mantegazza \footnote{Dipartimento di Matematica e Applicazioni, Universit\`a di Napoli, Via Cintia, Monte S. Angelo
80126 Napoli, Italy, c.mantegazza@sns.it}}

\begin{document}

\maketitle
\begin{abstract} We study some qualitative properties of ancient solutions of superlinear heat equations on a Riemannian manifold, with particular interest in positivity and constancy in space.
\end{abstract}


\section{Introduction}

In this paper we continue the study started in~\cite{CaMa1} about some properties of solutions of semilinear heat equations on a complete and connected Riemannian manifold $(M,g)$ without boundary. We will consider the model equations $u_t = \Delta u + |u|^p$ with $p>1$, where $\Delta$ is the Laplace--Beltrami operator of $(M,g)$. 

\begin{defn} We call a solution of the equation $u_t = \Delta u + |u|^p$
\begin{itemize}
\item {\em ancient} if it is defined in $M\times(-\infty,T)$ for some $T\in\R$,

\item {\em immortal} if it is defined in $M\times(T,+\infty)$ for some $T\in\R$,

\item {\em eternal} if it is defined in $M\times\R$.
\end{itemize}
We call a solution $u$ {\em trivial} if it is constant in space, that is, $u(x,t)=u(t)$ and solves the ODE $u'=|u|^p$. We say that $u$ is simply {\em constant} if it is constant in space and time.
\end{defn}

Notice that positive ancient (or negative immortal) trivial solutions always exists (the problem reduces to solve the above ODE). 

A reason for the interest in ancient or eternal solutions is that they typically arise as blow--up limits when the solutions of semilinear parabolic equations (in bounded intervals) develop a singularity at a certain time $T \in \mathbb{R}$, i.e. the solution $u$ becomes unbounded as $t \to T^-$. 

Our main result is the following theorem.

\begin{thm}\label{main}
Let $(M,g)$ be a n--dimensional compact Riemannian manifold without boundary such that $\Ric>0$. Let $u$ be an ancient solution to the semilinear heat equation $u_t = \Delta u + |u|^p$, with $1<p< \frac{n(n+2)}{(n-1)^2}$. Then $u$ is trivial.
\end{thm}

In Theorem 4.4 of~\cite{CaMa1} we proved triviality of ancient solutions for the case $p=2$ under a suitable growth assumption, hence, Theorem~\ref{main}, besides extending the conclusion to any exponent $1 < p < \frac{n(n+2)}{(n-1)^2}$, improves such result not requiring any growth assumptions on the solutions. This is achieved through a priori gradient and decay estimates of independent interest, which we discuss in Section~\ref{gradient}. Moreover, we underline that we do not assume the positivity of the solutions, but we just obtain it under the hypothesis of boundedness from below of the Ricci tensor (see Section~\ref{positivity}), thus, as a consequence, it is possible to improve results such as Theorem~1 in~\cite{quittner} or Corollary~1.6 in~\cite{merlezaag}.

In the Euclidean space, it is well known that nontrivial global radial (static) solutions on $\mathbb{R}^n \times \mathbb{R}$ exist for any supercritical exponent $p \geq \frac{n+2}{n-2}$. Conversely, while triviality of eternal {\em radial} solutions can be shown in the full range of subcritical exponents $1<p<\frac{n+2}{n-2}$, the same expected result for general (not necessarily radial) solutions is known only in the range $1<p< \frac{n(n+2)}{(n-1)^2}$ (as in Theorem~\ref{main} above), and it remains a challenging open problem when $\frac{n(n+2)}{(n-1)^2} \leq p< \frac{n+2}{n-2}$ (see~\cite{pqs1,pqs2}).

In Section~\ref{positivity} we show that the boundedness from below of the Ricci tensor of the manifold $(M,g)$ implies the positivity of ancient solutions. In Section~\ref{gradient} we obtain a universal a priori estimate which implies the decay at minus infinity of ancient solutions, as well as a gradient estimate of Li--Yau type. Finally, in Section~\ref{triviality}, we prove Theorem~\ref{main} as a corollary of a triviality result. 

\bigskip

{\em In all of the paper, the Riemannian manifolds $(M,g)$ will be smooth, complete, connected and without boundary. We will denote with $\Delta$ the associated Laplace--Beltrami operator. As it is standard, we will write $\Ric\geq \lambda$ with the meaning $\Ric\geq \lambda g$, that is, all the eigenvalues of the Ricci tensor are larger or equal than $\lambda\in\R$. Finally, all the solutions we will consider are classical, $C^2$ in space and $C^1$ in time.}

\section{Positivity}\label{positivity}

We start showing positivity of ancient solutions. We will also see that eternal solutions are trivial in the compact case.

\begin{prop}\label{p1}
Let $(M,g)$ be compact and $u$ an ancient solution of the equation $u_t = \Delta u + |u|^p$ with $p>1$ in $M\times (-\infty,T)$, for some $T\in\R$, then either $u\equiv 0$ or $u>0$ everywhere. 
\end{prop}

\begin{proof}
For every $t<T$, we define $x_t\in M$ as the point such that $u(x_t, t) = \min_{x\in M} u(x,t)$ and we set $v(t) = u(x_t, t)$, then, by maximum principle or more precisely, by Hamilton's trick (see~\cite{hamilton2} 
or~\cite[Lemma~2.1.3]{Manlib} for details), at almost every $t\in(-\infty,T)$ (precisely when $v'(t)$ exists -- notice that $v$ is locally Lipschitz, as $M$ is compact) there holds $v'(t) \geq\vert v(t)\vert^p$.

If $v(t_0)<0$ at some time $t_0\in(-\infty,T)$, then integrating the above differential inequality in intervals $[t_1,t_0]$, it is easy to see that, moving in the past, $v(t)$ goes to $-\infty$ in finite time, which is a contradiction. Thus, $v(t)\geq 0$ for every $t\in(-\infty,T)$, which implies $u\geq 0$ everywhere. Then, by strong maximum principle, either $u>0$ everywhere ($M$ is connected) or $u\equiv 0$ and we are done.
\end{proof}

\begin{cor}\label{coro1}
If $(M,g)$ is compact and $u$ is an eternal solution of $u_t = \Delta u + |u|^p$ with $p>1$ in $M\times\R$, then $u\equiv0$.
\end{cor}
\begin{proof}
By the previous proposition, if $u\not\equiv0$, then $u$ is positive everywhere and (with the same notation of the previous theorem) if $v(t_0)>0$, for some $t_0\in\R$, by integrating the differential inequality $v'(t) \geq v^p (t)$, we see that $v(t)$ goes to $+\infty$ in finite time, against the hypothesis that $u$ is an eternal solution.
\end{proof}

\begin{rem}
With the same argument, we can see that an immortal nonnegative solution is identically zero.
\end{rem}

\begin{rem}
In the noncompact situation, the conclusion of Corollary~\ref{coro1} does not necessarily hold. Consider $M = \R^n$ and $u:\R^n\to\R$ given by a ``Talenti's  function'' (an extremal of Sobolev inequality, see~\cite{tal} and also~\cite{aubinli}), that is,
$$
u(x)=\left(\,\frac{n(n-2)}{n(n-2) + \vert x \vert^2}\right)^{\frac{n-2}{2}}\,,
$$
which, by a straightforward computation, satisfies $\Delta u + u^{\frac{n+2}{n-2}} = 0$ in $\R^n$, for $n\geq 3$. In particular, $u$ is a nonzero eternal (static) solution for the semilinear heat equation $u_t=\Delta u + |u|^p$, with $p = \frac{n+2}{n-2}$.
\end{rem}

We deal now with the noncompact case, following the technique of~\cite[Proposition~2.1]{chen2}. 

\begin{lem}\label{p2} Let the Ricci tensor of the $n$--dimensional Riemannian manifold $(M,g)$ be uniformly bounded below by $-K(n-1)$, with $K\geq 0$, and let $u$ be a solution of the equation $u_t = \Delta u + |u|^p$ with $p>1$ in $M\times[0,T)$. For any $0<\delta<1$, there is a constant $C_\delta>0$ such that, if $u\geq -L$, for some positive value $L\in\R$, in the ball $B_{Ar_0}(x_0)$ at $t=0$, with 
$$
A\geq 4+2(n-1)T/r_0^2+2(n-1)T\sqrt{K}/r_0,
$$
then,
$$
u(x,t)\geq \min\left\{ - 
\Bigl(({1-\delta})({p-1}) t+ L^{1-p}\Bigr)^{\frac{1}{1-p}},\, - \frac{C_\delta}{(A r_0)^{\frac{2}{p-1}}}\right\}\,
$$
for every $x\in B_{Ar_0/4}(x_0)$ and $t\in[0,T)$.
\end{lem}

\begin{proof}  By the {\em Laplacian comparison} theorem (see~\cite[Chapter~9, Section~3.3]{petersen1} and also~\cite{sheng1}), if $\Ric\geq -K(n-1)$, with $K\geq 0$, we have
\begin{equation}\label{est1}
-\Delta d(x_0,x)\geq -\frac{n-1}{d(x_0,x)}-(n-1)\sqrt{K}\geq-\frac{n-1}{r_0}-(n-1)\sqrt{K}
\end{equation}
whenever $d(x_0,x)\geq r_0$, {\em in the sense of support functions} (or {\em in the sense of viscosity}, see~\cite{crisli1} -- check also~\cite[Appendix~A]{manmasura} for comparison of the two notions), in particular, this inequality can be used in maximum principle arguments, see again~\cite[Chapter~9, Section~3]{petersen1}, for instance. Hence, 
$$
-\Delta d(x_0,x)+\frac{n-1}{r_0}+(n-1)\sqrt{K}\geq 0\,,
$$
for every $x\in\ M$ such that $d(x_0,x)\geq r_0$.

We consider the function $w(x,t)=u(x,t)\psi(x,t)$ with 
$$
\psi(x,t)=\varphi\biggl(\frac{d(x_0,x)+\bigl(\frac{n-1}{r_0}+(n-1)\sqrt{K}\,\bigr)t}{Ar_0}\biggr)\,,
$$
where $\varphi:\R\to\R$ is a smooth, nonnegative and nonincreasing
function such that $\varphi=1$ in $(-\infty, 3/4]$ and
$\varphi=0$ in $[1,+\infty)$. Moreover, we ask that $\varphi$, $-\varphi^\prime$ and $\varphi^{\prime\prime}$ are all positive in a small interval $(1-\varepsilon,1)\subset(3/4,1)$.

For every $t\in[0,T)$, we let $w_{\min}(t)=\min_{x\in M} w(x,t)$ which is well defined (by construction, $\varphi(s)$ is zero for $s\geq 1$). Moreover, $w_{\min}$ is locally Lipschitz, hence differentiable at almost every time $t\in[0,T)$. If $w_{\min}(t)<0$, the minimum of $w(\cdot,t)$ is achieved at some point $x_t\in B_{Ar_0}(x_0)$, with $\psi(x_t,t)>0$ and $u(x_t,t)<0$. Hence, in a space--time neighborhood of $(x_t,t)$ we have $\varphi(\dots)=\psi>0$, so we can compute
\begin{align}
\Bigl(\frac{\partial}{\partial t}-\Delta\Bigr)w=
&\,\psi\Bigl(\frac{\partial}{\partial t}-\Delta\Bigr)u+u\Bigl(\frac{\partial}{\partial t}-\Delta\Bigr)\psi-2\nabla\psi\nabla u\\
=&\,\psi |u|^p+
\varphi'(\dots)\frac{u}{Ar_0}\Bigl[-\Delta d(x_0,\cdot)+\frac{n-1}{r_0}+(n-1)\sqrt{K}\Bigr]\\
&\,-\varphi''(\dots)\frac{u}{A^2r_0^{2}} -2\nabla\psi\nabla u\\
=&\,\psi |u|^p+
\varphi'(\dots)\frac{u}{Ar_0}\Bigl[-\Delta d(x_0,\cdot)+\frac{n-1}{r_0}+(n-1)\sqrt{K}\Bigr]\\
&\,-\varphi''(\dots)\frac{u}{A^2r_0^{2}} -2\frac{\nabla\psi\nabla w}{\psi} +2u\frac{\vert\nabla\psi\vert^2}{\psi}\\
=&\,\psi |u|^p+
\varphi'(\dots)\frac{u}{Ar_0}\Bigl[-\Delta d(x_0,\cdot)+\frac{n-1}{r_0}+(n-1)\sqrt{K}\Bigr]\\ 
&\,-2\frac{\nabla\psi\nabla w}{\psi}+\frac{u}{A^2r_0^{2}}\Bigl(\frac{2[\varphi'(\dots)]^2}{\varphi(\dots)}-\varphi''(\dots)\Bigr)\,,\label{eqqq1}
\end{align}
at the smooth points of the function $\psi$, in particular, at the smooth points of the distance function  (that is, the points not belonging to the {\em cutlocus} of $x_0$). Notice that 
we used the fact that $\vert\nabla d(x_0,\cdot)\vert^2=1$.

If $x_t\in B_{2r_0}(x_0)$, we have
\begin{align*}
\frac{d(x_0,x_t)+\bigl(\frac{n-1}{r_0}+(n-1)\sqrt{K}\bigr)t}{Ar_0}
\leq&\,\frac{2r_0+\bigl(\frac{n-1}{r_0}+(n-1)\sqrt{K}\bigr)T}{Ar_0}\\
\leq&\,\frac{2r_0+{(n-1)T}/{r_0}+(n-1)\sqrt{K}\,T}{4r_0+2(n-1)T/r_0+2(n-1)\sqrt{K}\,T} = 1/2\,,
\end{align*}
by the hypothesis on $A$.\\
Thus, by the choice of $\varphi$, the function $\psi$ is locally constant equal to $1$ around $(x_t,t)$, hence smooth and  $\nabla\psi=\Delta\psi=\frac{\partial\psi}{\partial t}=0$. It follows, by the first line in computation~\eqref{eqqq1}, that in such case there holds
\begin{equation}\label{spec}
\Bigl(\frac{\partial}{\partial t}-\Delta\Bigr)w=\psi |u|^p\,,
\end{equation}
locally around $(x_t,t)$.

If instead $d(x_t,x_0)\geq 2r_0$, estimate~\eqref{est1} and formula~\eqref{eqqq1} hold in sense of support functions (or of viscosity), locally around $(x_t,t)$. Moreover, $\varphi'(\dots)u\geq0$, as $u$ is locally negative, that is, the factor in front of the second term in the right hand side of formula~\eqref{eqqq1} is nonnegative. Hence, locally we have
\begin{align*}
\Bigl(\frac{\partial}{\partial t}-\Delta\Bigr)w\geq&\,\psi |u|^p -2\frac{\nabla\psi\nabla w}{\psi}+\frac{u}{A^2r_0^{2}}\Bigl(\frac{2[\varphi'(\dots)]^2}{\varphi(\dots)}-\varphi''(\dots)\Bigr)\\
\geq&\,\psi |u|^p -2\frac{\nabla\psi\nabla w}{\psi}-\frac{\vert u\vert}{A^2r_0^{2}}\biggl\vert\frac{2[\varphi'(\dots)]^2}{\varphi(\dots)}-\varphi''(\dots)\biggr\vert\,.
\end{align*}
Notice that, by the argument leading to equation~\eqref{spec}, this conclusion holds also when $x_t\in B_{2r_0}(x_0)$, hence, independently of the position of the point $x_t$, for every time $t\in[0,T)$ such that $w_{\min}(t)<0$.

Then, by maximum principle (Hamilton's trick) and standard support functions (or viscosity) techniques, if $w^\prime_{\min}(t)$ exists, we get the inequality
$$
w^\prime_{\min}(t)\geq\varphi(\dots) |u|^p-\frac{\vert u\vert}{A^2r_0^{2}}\biggl\vert\frac{2[\varphi'(\dots)]^2}{\varphi(\dots)}-\varphi''(\dots)\biggr\vert\,,
$$
with the right hand side evaluated at $(x_t,t)$.\\
Now, it is not difficult to see that, by our assumptions on the function $\varphi:\R\to\R$, there exists a positive constant $C$ such that $\Bigl\vert\frac{2[\varphi']^{2}}{\varphi}-\varphi''\Bigr\vert\leq C \varphi^{1/p}$, as $p>1$. Hence, simplifying the notation, 
for any $\delta\in(0,1)$, there holds
\begin{align*}
w^\prime_{\min}\geq&\,\varphi|u|^p-\frac{C\vert u\vert \varphi^{1/p}}{A^2r_0^{2}}\\
\geq &\,\varphi |u|^p - \frac{\delta}{p}\varphi |u|^p - \frac{(p-1)\,C^{\frac{p}{p-1}}}{p \delta^{\frac{1}{p-1}}  A^{\frac{2p}{p-1}} {r_0}^{\frac{2p}{p-1}}}\\
=&\,\frac{\psi^p|u|^p}{\psi^{p-1}}\bigl(1-\delta/p\bigr) - \frac{C_p}{\delta^{\frac{1}{p-1}}  A^{\frac{2p}{p-1}} {r_0}^{\frac{2p}{p-1}}}\\
=&\,\frac{|w_{\min}|^{p}}{\psi^{p-1}} (1-\delta/p) - \frac{C_p}{\delta^{\frac{1}{p-1}} A^{\frac{2p}{p-1}} {r_0}^{\frac{2p}{p-1}}}\,,
\end{align*}
where we used Young inequality, for some positive constant $C_p$.\\
As $0<\psi(x_t,t)\leq 1$, when $w_{\min}(t)<0$, we get 
\begin{align}
w^\prime_{\min}
\geq &\,|w_{\min}|^{p} (1-\delta) +\frac{\delta(p-1)}{p}|w_{\min}|^{p}- \frac{C_p}{\delta^{\frac{1}{p-1}} A^{\frac{2p}{p-1}} {r_0}^{\frac{2p}{p-1}}}\\
\geq &\, |w_{\min}|^{p}(1-\delta)+\frac{\delta(p-1)}{p}\Bigl(|w_{\min}|^{p}-\frac{C_\delta^p}{A^{\frac{2p}{p-1}} {r_0}^{\frac{2p}{p-1}}}\Bigr)\label{est2}\,,
\end{align}
for some positive constant $C_\delta$, at almost every time $t\in[0,T)$ such that $w_{\min}(t)<0$. 

Resuming, for almost every $t\in[0,T)$, either $w_{\min}(t)$ is nonnegative or the differential inequality~\eqref{est2} holds, moreover $w_{\min}(0)\geq -L$, by hypothesis (recall that if $w_{\min}(0)<0$ any point of minimum of $w(\cdot,0)$ belongs to $B_{Ar_0}(x_0)$). By an easy ODE's argument, integrating such differential inequality, we conclude
$$
w_{\min}(t)\geq \min\biggl\{ - 
\Bigl(({1-\delta})({p-1}) t+ L^{1-p}\Bigr)^{\frac{1}{1-p}},\, - \frac{C_\delta}{(A r_0)^{\frac{2}{p-1}}}\biggr\}\,,
$$
for every $t\in[0,T)$. It follows that 
$$
u(x,t)=u(x,t)\psi(x,t)=w(x,t)\geq \min\biggl\{ - 
\Bigl(({1-\delta})({p-1}) t+ L^{1-p}\Bigr)^{\frac{1}{1-p}},\, - \frac{C_\delta}{(A r_0)^{\frac{2}{p-1}}}\biggr\}\,,
$$
for every $x\in B_{Ar_0/4}(x_0)$ and $t\in[0,T)$.
\end{proof}

As a consequence of this lemma, we prove the positivity of ancient solutions. Notice that we do not ask any bound on $u$.

\begin{thm}\label{p3} Let the Ricci tensor of $(M,g)$ be uniformly bounded below. If $u$ is an ancient solution of the equation $u_t = \Delta u + |u|^p$ with $p>1$ in $M\times(-\infty,T)$, then either $u\equiv 0$ or $u>0$ everywhere.
\end{thm}

\begin{proof} We only need to show that $u\geq 0$ everywhere, then the conclusion will follow by the strong maximum principle, as before.

Let the Ricci tensor of $(M,g)$ be bounded below by $-K(n-1)$, for some $K\geq 0$. Since the estimate in the previous lemma is invariant by translation in time, for every $m\in\NN$, we can consider the interval $[-m,T)$, for any $m>-T$, and conclude that
$$
u(x,t)\geq \min\left\{ - 
\Bigl(({1-\delta})({p-1}) (t+m)+ L^{1-p}\Bigr)^{\frac{1}{1-p}},\, - \frac{C_\delta}{(A r_0)^{\frac{2}{p-1}}}\right\}\,,
$$
for every $x\in B_{Ar_0/4}(x_0)$ and $t\in[-m,T)$, with $-L\leq\inf_{x\in B_{Ar_0}(x_0)} u(x,-m)$ and 
$$
A\geq 4+2(n-1)(T+m)/r_0^2+2(n-1)(T+m)\sqrt{K}/r_0\,.
$$
In particular, for every $t\in[-m+1,T)$ and $x\in B_{Ar_0/4}(x_0)$, sending $L$ to $+\infty$, we get
$$
u(x,t)\geq \min\left\{ - 
\Bigl(({1-\delta})({p-1})(t+m)\Bigr)^{\frac{1}{1-p}},\, - \frac{C_\delta}{(A r_0)^{\frac{2}{p-1}}}\right\}\,.
$$
Sending now $A\to+\infty$, we have that for every $t\in[-m+1,T)$ and $x\in M$ there holds
$$
u(x,t)\geq -\Bigl(({1-\delta})({p-1})(t+m)\Bigr)^{\frac{1}{1-p}}\,,
$$
for every $m\in\NN$, large enough. Sending finally $m\to+\infty$, we conclude that $u\geq 0$ everywhere.
\end{proof}

\begin{cor}\label{p4} Let the Ricci tensor of $(M,g)$ be uniformly bounded below. If $u$ is a solution of the semilinear elliptic equation $\Delta u + |u|^p=0$ with $p>1$ in $M$, then either $u\equiv 0$ or $u>0$ everywhere.
\end{cor}

\section{Gradient and decay estimates}\label{gradient}

We now show a gradient estimate for positive solutions of the semilinear heat equation $u_t = \Delta u +|u|^p$ on manifolds with nonnegative Ricci tensor. A similar result for the classical heat equation has been proven by Souplet and Zhang in~\cite{souzha}. We will then apply this estimate in order to obtain the triviality of ancient solutions of $u_t = \Delta u + |u|^p$  under some hypotheses. 

\begin{lem}\label{szsem}
Let $(M,g)$ be an $n$--dimensional Riemannian manifold with $\Ric\geq K(n-1)$, for some $K\geq 0$ and let $p>1$. Let $u$ be a positive solution of the semilinear heat equation $u_t = \Delta u + |u|^p$ in $Q_{R,T} = B(x_0,R) \times [T_0 - T, T_0]$, with $B(x_0,R)$ the geodesic ball centered at $x_0$ of radius $R$ in $M$. Assume that $u \leq D$ in $Q_{R,T}$. Then, there exists a constant $C=C(n,p)$ such that on $Q_{R/2,T/4}$ there holds
\begin{equation}\label{liyausem}
\frac{|\nabla u(x,t)|}{u(x,t)} \leq C \left( \frac{1}{R} + \frac{1}{\sqrt{T}} + \sqrt{\bigl(pD^{p-1}-(n-1)K\bigr)_+} \right) \left( 1 + \log \frac{D}{u(x,t)} \right),
\end{equation}
where $\bigl(pD^{p-1}-(n-1)K\bigr)_+$ denotes $\max\,\bigl\{pD^{p-1}-(n-1)K,0\bigr\}$.
\end{lem}

\begin{proof}
Let us define
\begin{equation}\label{few}
f= \log \left( \frac{u}{D} \right), \qquad w= \frac{|\nabla f|^2}{(1-f)^2}.
\end{equation}  
Thanks to the semilinear heat equation we easily see that
\begin{equation}\label{eqf}
f_t = \Delta f + |\nabla f|^2 + (D e^{f})^{p-1},
\end{equation}
then, in an orthonormal basis, we have
\begin{align*}
w_t =&\, \frac{2 \nabla f \nabla f_t}{(1-f)^2} + \frac{2 |\nabla f|^2 f_t}{(1-f)^3}\\
=&\,\frac{2 \nabla f \nabla (\Delta f + |\nabla f|^2 + (D e^f)^{p-1})}{(1-f)^2} + \frac{2 |\nabla f|^2 (\Delta f + |\nabla f|^2 + (D e^f)^{p-1})}{(1-f)^3}\\
=&\,\frac{2 \nabla f \nabla (\Delta f + |\nabla f|^2)}{(1-f)^2} + \frac{2 |\nabla f|^2 (\Delta f + |\nabla f|^2)}{(1-f)^3} + \frac{2 |\nabla f|^2 (D e^f)^{p-1}}{(1-f)^2} \left[ (p-1) + \frac{1}{1-f}\right]\\
=&\,\frac{2 f_{jii}f_j + 4 f_i f_j f_{ij}}{(1-f)^2} + \frac{2 f_{i}^{2} f_{jj} + 2|\nabla f|^4}{(1-f)^3} + \frac{2 |\nabla f|^2 (D e^f)^{p-1}}{(1-f)^2} \left[ (p-1) + \frac{1}{1-f} \right]\\
=&\,\frac{2 f_{iij}f_j -2\Ric_{ij}f_if_j+ 4 f_i f_j f_{ij}}{(1-f)^2} + \frac{2 f_{i}^{2} f_{jj} + 2|\nabla f|^4}{(1-f)^3} + \frac{2 |\nabla f|^2 (D e^f)^{p-1}}{(1-f)^2} \left[ (p-1) + \frac{1}{1-f} \right],
\end{align*}
where we interchanged derivatives (hence, there is an ``extra'' error term given by the Ricci tensor), passing from the fourth to the fifth line and we used the usual convention of summing on repeated indexes.\\
Now,
\begin{equation}\label{sz2}
\nabla_jw =\nabla_j \left( \frac{f_i^2}{(1-f)^2} \right)  =
\frac{2 f_i f_{ji}}{(1-f)^2} + \frac{2f_{i}^{2} f_{j}}{(1-f)^3}
\end{equation}
and
\begin{equation}\label{sz3}
\Delta w = \frac{2 f_{ij}^2 }{(1-f)^2} + \frac{2 f_i f_{jji}}{(1-f)^2} + \frac{8 f_i f_{ij} f_j }{(1-f)^3} + \frac{2f_{i}^{2} f_{jj}}{(1-f)^3} 
+ \frac{6f_{i}^{2} f_{j}^{2}}{(1-f)^4}.
\end{equation}
Hence, we get
\begin{align*}
w_t -\Delta w 
=&\,\frac{2 f_j f_{iij} -2\Ric_{ij}f_if_j+ 4 f_i f_{ij}f_j}{(1-f)^2} + \frac{2 f_{i}^{2} f_{jj} + 2|\nabla f|^4}{(1-f)^3} \\
&\,+ \frac{2 |\nabla f|^2 (D e^f)^{p-1}}{(1-f)^2}\left[ (p-1) + \frac{1}{1-f} \right]\\
&\,-\frac{2 f_{ij}^2 }{(1-f)^2} - \frac{2 f_i f_{jji}}{(1-f)^2} - \frac{8 f_i f_{ij} f_j }{(1-f)^3} - \frac{2f_{i}^{2} f_{jj}}{(1-f)^3} 
- \frac{6 f_{i}^{2} f_{j}^{2}}{(1-f)^4}\\
=&\,\frac{4 f_j f_{ij}f_j-2\Ric_{ij}f_if_j}{(1-f)^2} + \frac{2|\nabla f|^4}{(1-f)^3} + \frac{2 |\nabla f|^2 (D e^f)^{p-1}}{(1-f)^2} \left[ (p-1) + \frac{1}{1-f} \right]\\
&\,-\frac{2 f_{ij}^2 }{(1-f)^2}  - \frac{8 f_i f_{ij} f_j }{(1-f)^3} - \frac{6 f_{i}^{2} f_{j}^{2}}{(1-f)^4}.
\end{align*}
As by hypothesis, $f \leq 0$, we have 
$$
\left[ (p-1) + \frac{1}{1-f} \right]e^{(p-1) f} \leq p,
$$
hence, since $\Ric_{ij} f_i f_j \geq K(n-1) |\nabla f|^2$, we get
\begin{align*}
w_t -\Delta w
&\,\leq \frac{4 f_i f_{ij}f_j}{(1-f)^2} + \frac{2|\nabla f|^4}{(1-f)^3} + \frac{2\bigl(pD^{p-1}-(n-1)K\bigr)|\nabla f|^2}{(1-f)^2}\\
&\,-\frac{2 f_{ij}^2 }{(1-f)^2}  - \frac{8 f_i f_{ij} f_j }{(1-f)^3} - \frac{6 \vert\nabla f\vert^4}{(1-f)^4}.
\end{align*}
For the sake of simplicity let us set $L=\bigl(pD^{p-1}-(n-1)K\bigr)$.
Notice that by~\eqref{sz2}, there holds
$$
\langle\nabla f\,\vert\, \nabla w\rangle = \frac{2 f_i f_{ij} f_j }{(1-f)^2} + \frac{2\vert\nabla f\vert^{4} }{(1-f)^3},
$$
hence, substituting, we get
\begin{align}
w_t -\Delta w\leq 
&\,2\langle\nabla f\,\vert\, \nabla w\rangle -\frac{2\vert\nabla f\vert^{4} }{(1-f)^3} + \frac{2L|\nabla f|^2}{(1-f)^2} -\frac{2 f_{ij}^2 }{(1-f)^2}  - \frac{8 f_i f_{ij} f_j }{(1-f)^3} - \frac{6 \vert\nabla f\vert^4}{(1-f)^4}\nonumber\\
=&\, \left(2 -\frac{2}{1-f} \right) \langle\nabla f\,\vert\, \nabla w\rangle+ \frac{2L|\nabla f|^2}{(1-f)^2}-\frac{2\vert\nabla f\vert^{4} }{(1-f)^3} -\frac{2 f_{ij}^2 }{(1-f)^2}- \frac{4 f_i f_{ij} f_j }{(1-f)^3} - \frac{2 \vert\nabla f\vert^4}{(1-f)^4}\nonumber\\
=&\,-\frac{2f}{1-f}\langle\nabla f\,\vert\, \nabla w\rangle+ \frac{2L|\nabla f|^2}{(1-f)^2} -\frac{2\vert\nabla f\vert^{4} }{(1-f)^3}-\frac{2}{(1-f)^2}\left(f_{ij}+\frac{f_if_j}{1-f}\right)^2\nonumber\\
\leq&\,-\frac{2f}{1-f}\langle\nabla f\,\vert\, \nabla w\rangle+ 2Lw-2(1-f)w^2.\label{sz7}
\end{align}

We introduce the following cut--off functions (of Li and Yau~\cite{liyau}). Let $\psi$ be a smooth function supported in $Q_{R,T}$ with the following properties:
\begin{enumerate}
\item $\psi(x,t)=\varphi(d^M(x_0,x),t)\in[0,1]$ with $\varphi(r,t) \equiv 1$ if $r\leq R/2$ and $T_0 -T/4 \leq t \leq T_0$,
\item $\varphi$ is nonincreasing in the space variable $r$,
\item $|\nabla\psi|/\psi^a=|\partial_r \varphi|/\varphi^a \leq C_a /R$ and $|\partial^2_{rr} \varphi|/\varphi^a \leq C_a /R^2$, when $0<a<1$,
\item $|\partial_t \psi|/\psi^{1/2} \leq C / T$,
\end{enumerate}
for some constants $C$, $C_a$ independent of $R$ and $T$.

Then, by inequality~\eqref{sz7} with a straightforward calculation, setting $b = - \frac{2 f}{1-f} \nabla f$ one has
\begin{equation}\label{sz8}
\begin{split}
\Delta (\psi w) &+ \langle b\,\vert\, \nabla (\psi w)\rangle - 2 \Bigl\langle\frac{\nabla \psi}{\psi}\,\Bigr\vert\,\nabla (\psi w)\Bigr\rangle - (\psi w)_t\\
&\geq 2 \psi (1-f) w^2 + \langle b \,\vert\,\nabla \psi\rangle w  - 2 \frac{|\nabla \psi|^2}{\psi} w + w\Delta \psi - \psi_t w - 2L w \psi. 
\end{split}
\end{equation}

Suppose that the positive maximum of $\psi w$ is reached at some point $(x_1,t_1)\in Q_{R,T}$, which cannot be on the boundary where $\psi=0$. Arguing again (as in Lemma~\ref{p2}) in the sense of support function, if necessary, at such maximum point there holds $\Delta (\psi w) \leq 0$, $(\psi w)_t = 0$ and $\nabla (\psi w) = 0$,  hence
\begin{equation}\label{sz9}
2 \psi (1-f) w^2(x_1,t_1)\leq-\left[ \langle b\,\vert\, \nabla \psi\rangle w  - 2 \frac{|\nabla \psi|^2}{\psi} w + (\Delta \psi) w - \psi_t w - 2L w \psi \right] (x_1,t_1).
\end{equation}

We now estimate each term on the right hand side. For the first term we have,
\begin{align}
|\langle b\,\vert\, \nabla \psi\rangle w | 
\leq&\, \frac{2 |f|}{1-f} |\nabla f|\,|\nabla \psi|w\\
=&\, 2 w^{3/2} |f| |\nabla \psi|\\
=&\, 2 [(1-f)\psi w^2]^{3/4} \frac{|f| |\nabla \psi|}{[(1-f)\psi]^{3/4}}\\
\leq&\, (1-f)\psi w^2 + C \frac{(f |\nabla \psi|)^4}{[(1-f)\psi]^{3}}\\
\label{sz10}\leq&\,(1-f) \psi w^2 + C \frac{f^4}{R^4 (1-f)^{3}},
\end{align}
by the properties of the function $\psi$.\\
For the second term,
\begin{equation}\label{sz11}
\frac{|\nabla \psi|^2}{\psi} w = \psi^{1/2}\frac{|\nabla \psi|^2}{\psi^{3/2}}w \leq \frac18 \psi w^2 + C \left( \frac{|\nabla \psi|^2}{\psi^{3/2}} \right)^2 \leq \frac18 \psi w^2 + \frac{C}{R^{4}}.
\end{equation}
Thanks to the assumption on the nonnegative Ricci curvature, by the {\em Laplacian comparison} theorem (see formula~\eqref{est1}), we have
\begin{align}
- (\Delta \psi) w 
\leq&\, - \left[\partial^2_{rr} \varphi + \frac{n-1}{r}\partial_r \varphi \right]w\\
\leq&\, \left[|\partial^2_{rr} \varphi| + \frac{2(n-1)}{R}|\partial_r \varphi|\right] w\\
\leq&\, \varphi^{1/2} w \left(\frac{|\partial^2_{rr} \varphi|}{\varphi^{1/2}} + 2 (n-1) \frac{|\partial_r \varphi|}{R \varphi^{1/2}} \right)\\
\leq&\, \frac18 \varphi w^2 + C \left( \left[ \frac{|\partial^2_{rr} \varphi|}{\varphi^{1/2}} \right]^2 + \left[ \frac{|\partial_r \varphi|}{R \varphi^{1/2}} \right]^2 \right)\\
\label{sz12}\leq&\, \frac18 \psi w^2 + \frac{C}{R^{4}},
\end{align}
by the properties of the functions $\varphi$ (we recall that $\partial_r\varphi\leq 0$) and $\psi$.\\
Now we estimate $|\psi_{t}| w$ as
\begin{equation}\label{sz13}
|\psi_{t}| w = \psi^{1/2}\frac{|\psi_{t}|}{\psi^{1/2}}w \leq \frac18 \psi w^2 + C \left( \frac{|\psi_{t}|}{\psi^{1/2}} \right)^2\leq \frac18 \psi w^2 + \frac{C}{T^{2}},
\end{equation}
again by the properties of $\psi$.\\
Finally, we deal with the last term, there holds
\begin{equation}\label{sz14}
2 L w \psi \leq 2 L_+ w \psi \leq \frac18 \psi w^2 + C L_{+}^{2},
\end{equation}
as $\psi\leq 1$.\\
Substituting estimates~\eqref{sz10},~\eqref{sz11},~\eqref{sz12},~\eqref{sz13},~\eqref{sz14} in the right hand side of inequality~\eqref{sz9}, we deduce
$$
2 (1-f) \psi w^2 \leq (1-f) \psi w^2 + C \frac{f^4}{R^4 (1-f)^{3}} + \frac12 \psi w^2 + \frac{C}{R^{4}} + \frac{C}{T^{2}} + C L_{+}^{2}.
$$
Recalling that $f \leq 0$, it follows
$$
\psi w^2 (x_1,t_1) \leq C \frac{f^4}{R^4 (1-f)^{4}} + \frac12 \psi w^2 (x_1,t_1) + \frac{C}{R^{4}} + \frac{C}{T^{2}} + C L_{+}^{2}
$$
and, since $f^4/(1-f)^4 \leq 1$, we conclude that 
$$
\psi^2 (x,t) w^2 (x,t) \leq \psi^2 (x_1,t_1) w^2 (x_1,t_1) \leq  \psi (x_1,t_1) w^2 (x_1,t_1) \leq \frac{C}{R^{4}} + \frac{C}{T^{2}} + C L_{+}^{2},
$$
for all $(x,t)\in Q_{R,T}$.

As $\psi=1$ in $Q_{R/2,T/4}$, $w = |\nabla f|^2/(1-f)^2$ and $L=\bigl(pD^{p-1}-(n-1)K\bigr)$, we finally have
$$
\frac{|\nabla f|}{(1-f)} \leq \frac{C}{R} + \frac{C}{\sqrt{T}} + C \sqrt{\bigl(pD^{p-1}-(n-1)K\bigr)_+}
$$
for every $(x,t)\in Q_{R/2,T/4}$. Since $f = \log (u/D)$, we are done.

The constant $C$ can be effectively traced back and made explicit from estimates~\eqref{sz10} onward: it comes from reiterated applications of Young's inequality with numerical constants, from the properties~3 and~4 of the cut--off functions $\psi$ through the constants $C_a$ and $C$ respectively, and from the constant $2(n-1)$ appearing in the Laplacian comparison theorem used in 
estimate~\eqref{sz12}. So $C(n,p)$ is depending only on the dimension $n$ of the manifold $M$ and on the exponent $p>1$.
\end{proof}

\begin{rem} Notice that if $K>0$, then the manifold is compact, by Bonnet--Myers theorem (see~\cite{gahula}).
\end{rem}

If $u \leq D$ uniformly in $M\times [T_0 - T, T_0]$, then estimate~\eqref{liyausem} holds for every $(x,t)\in M\times [T_0-T/4, T_0]$, when  $R$ is large enough, hence, sending $R\to+\infty$, we get the following corollary.

\begin{cor}\label{szse2}
Let $(M,g)$ be an $n$--dimensional Riemannian manifold such that $\Ric\geq K(n-1)$, for some $K\geq0$. Let $u$ be a positive solution of the semilinear heat equation 
$u_t = \Delta u + |u|^p$ in $M\times [T_0 - T, T_0]$, for $p>1$. Assume that $u \leq D$, then, there exists $C=C(n,p)$ such that 
\begin{equation}\label{liyausem2}
\frac{|\nabla u(x,t)|}{u(x,t)} \leq C \left(\frac{1}{\sqrt{T}} + \sqrt{\bigl(pD^{p-1}-(n-1)K\bigr)_+} \right) \left( 1 + \log \frac{D}{u(x,t)} \right)
\end{equation}
for every $(x,t)$ in $M\times[T_0-T/4,T_0]$.
\end{cor}

\begin{rem}
If $M$ is compact, one can also prove this corollary directly following the proof of Lemma~\ref{szsem}, simply considering functions $\psi$ which are constant in space.
\end{rem}

\begin{cor}\label{szsem4}
Let $(M,g)$ be an $n$--dimensional  Riemannian manifold such that $\Ric \geq K(n-1)$, for some $K\geq 0$ and let $p>1$. Let $u$ be a nonzero ancient solution of the semilinear heat equation $u_t = \Delta u + |u|^p$ in 
$M\times (-\infty, T_0]$ bounded by $D>0$. Then, there exists a constant $C=C(n,p)$ such that for every $(x,t)\in M\times (-\infty, T_0]$ there holds
\begin{equation}\label{liyausem5}
\frac{|\nabla u(x,t)|}{u(x,t)} \leq C \sqrt{\bigl(pD^{p-1}-(n-1)K\bigr)_+} \left( 1 + \log \frac{D}{u(x,t)} \right),
\end{equation}
where $\bigl(pD^{p-1}-(n-1)K\bigr)_+$ denotes $\max\,\bigl\{pD^{p-1}-(n-1)K,0\bigr\}$.
\end{cor}
\begin{proof}
By Theorem~\ref{p3}, we know that if $u$ is nonzero then $u$ is necessarily positive. Then, since estimate~\eqref{liyausem} holds for every $(x,t)\in M\times (-\infty, T_0]$, for $R$ and $T$ large enough, sending $R,T\to+\infty$, we get the conclusion.
\end{proof}

Following Pol\'a\v cik--Quittner--Souplet~\cite{pqs2} for the Euclidean case, we now prove a universal a priori estimate for ancient solutions of $u_t = \Delta u + |u|^p$, with $1< p < \frac{n(n+2)}{(n-1)^2}$  on an $n$--dimensional Riemannian manifold $(M,g)$ with {\em bounded geometry} (injectivity radius positively uniformly bounded below and uniformly bounded Riemann tensor with all its covariant derivatives), for instance when $M$ is compact. As a corollary, such ancient solutions decay to zero at minus infinity.

\begin{prop}\label{universal2}
Let $(M,g)$ be an $n$--dimensional Riemannian manifold with bounded geometry. Let $u$ be a uniformly bounded below solution of the equation $u_t = \Delta u + |u|^p$ in $M\times (T_0,T)$, with $1< p < \frac{n(n+2)}{(n-1)^2}$. Then there exists a universal constant $C = C (n,p)$ such that there holds
\begin{equation}\label{univest}
u(x,t) + |\nabla u(x,t)|^{\frac{2}{p+1}} \leq C\left[ |t-T_0|^{-\frac{1}{p-1}}+|T-t|^{-\frac{1}{p-1}}\right]\,, 
\end{equation}
for every $(x,t) \in M\times (T_0,T)$.
\end{prop}

\begin{proof}[Proof (Sketch).] One can argue precisely in the same way as in the proof of Theorem~3.1 (ii) in~\cite{pqs2}, concerning the case $M=\R^n$. In such paper, assuming the conclusion fails, by a 
blow--up argument the authors can produce an eternal, nonzero, bounded and positive (classical) solution of equation $u_t = \Delta u + u^p$ in the whole $\R^n$, which is known it does not exist (see Theorem~A in the same paper). In our case of a manifold, the only difference in taking such blow--up (which is given by an appropriate ``rescaling'' of the solution) is that we have to ``rescale'' (dilate) also the ambient space $(M,g)$. Anyway, by the hypothesis of bounded Riemann tensor and injectivity radius uniformly bounded below, the dilated manifolds (together with the rescaled functions defined on them) smoothly converge, up to a subsequence, in the sense of Cheeger--Gromov (see~\cite{petersen1}, for instance) to $\R^n$ with its canonical flat metric. Hence, as for the case $M=\R^n$ discussed in~\cite{pqs2}, we obtain a nonnegative solution $u_t = \Delta u + u^p$ in the whole $\R^n$, with the same properties as above, which is a contradiction.
\end{proof}

As a corollary, sending $T_0\to-\infty$, we have the following decay estimate for ancient solutions, which must be nonnegative (hence bounded below) by Theorem~\ref{p3}.

\begin{cor}\label{decay}
Let $(M,g)$ be an $n$--dimensional Riemannian manifold with bounded geometry. Let $u$ be an ancient solution of the equation $u_t = \Delta u + |u|^p$ in $M\times (-\infty,T)$, with $1< p <  \frac{n(n+2)}{(n-1)^2}$, for some $T\in\R$. Then there exists a universal constant $C = C (n,p)$ such that
\begin{equation}\label{decayest}
u(x,t) \leq C (T-t)^{-\frac{1}{p-1}} \; \text{ for any } (x,t) \in M\times (-\infty,T).
\end{equation}
In particular 
$$
\lim_{t\to-\infty}\max_{x\in M}u(x,t)=0.
$$
\end{cor}

If the solution is actually eternal, we can also send $T\to+\infty$, concluding that $u\equiv0$.

\begin{cor}\label{eter}
Let $(M,g)$ be a $n$--dimensional Riemannian manifold with bounded geometry. Every eternal solution of the equation $u_t = \Delta u + |u|^p$ in $M\times\R$, with $1< p <  \frac{n(n+2)}{(n-1)^2}$, is identically zero.
\end{cor}

\begin{rem} The conclusion of Proposition~\ref{universal2}, hence also of Corollaries~\ref{decay} and~\ref{eter}, is conjectured to hold for the whole range of exponents $p\in\bigl(1, \frac{n+2}{n-2}\bigr)$, if $n\geq 3$, or $p\in(1, +\infty)$, when $n=1,2$. This would follow by the nonexistence of eternal, positive, classical solutions of equation $u_t = \Delta u + u^p$ in the whole $\R^n$, which at the moment can be shown in general only when $p < \frac{n(n+2)}{(n-1)^2}$ and assuming that such solutions are {\em radial} when $p\in\bigl[\frac{n(n+2)}{(n-1)^2},\frac{n+2}{n-2}\bigr)$ (see the discussion in~\cite[Section~2]{pqs2} and the references therein).
\end{rem}

\section{Triviality}\label{triviality}

We can now prove a triviality result and, consequently, Theorem~\ref{main} stated in the introduction.

\begin{thm}\label{liouanc}
Let $(M,g)$ be an $n$--dimensional compact Riemannian manifold such that $\Ric>0$. Let $u$ be an ancient solution to the semilinear heat equation $u_t = \Delta u + |u|^p$, with $p>1$, such that 
\begin{equation}\label{grow}
\lim_{t\to-\infty}\max_{x\in M}u(x,t) = 0,
\end{equation}
then $u$ is trivial.
\end{thm}
\begin{proof} As the Ricci tensor is positive and $M$ compact, $\Ric\geq (n-1)Kg$, for some $K>0$. By Theorem~\ref{p3} we then know that, if $u$ is nonzero, then $u$ is necessarily positive. Under the above growth hypothesis, there exists $T_0\in\R$ such that 
$$
0<u(x,t)\leq\bigl[(n-1)K/p\bigr]^{\frac{1}{p-1}}
$$
for every $x\in M$ and $t\leq T_0$, hence, by the estimate~\eqref{liyausem5}, we get (with a constant $C$ depending only on the dimension $n$ of the manifold $M$ and on the exponent $p>1$),
\begin{equation*}
\frac{|\nabla u(x,t)|}{u(x,t)} \leq C \sqrt{(pD^{p-1}-(n-1)K)_+} \left( 1 + \log \frac{D}{u(x,t)} \right)=0,
\end{equation*}
for every $(x,t)\in M\times(-\infty,T_0]$, as $D = \max_{M \times (-\infty,T_0]} u(x,t)\leq \bigl[(n-1)K/p\bigr]^{\frac{1}{p-1}}$.\\
Being $|\nabla u(x,t)| = 0$ for every $(x,t)\in M\times(-\infty, T_0]$, the function $u$ is constant in space for every $t\leq T_0$. By uniqueness of solutions ($M$ is compact), $u$ is trivial.
\end{proof}

\begin{rem}
The hypothesis~\eqref{grow} can be slightly weakened as follows. If $\Ric\geq (n-1)Kg$, it is sufficient that 
\begin{equation*}
\limsup_{t\to-\infty}\max_{x\in M}u(x,t)<\bigl[(n-1)K/p\bigr]^{\frac{1}{p-1}}.
\end{equation*}
\end{rem}

\begin{proof}[Proof of Theorem~\ref{main}]
The conclusion follows immediately by putting together Theorem~\ref{liouanc} with the decay estimate of Corollary~\ref{decay}.
\end{proof}

\bibliographystyle{amsplain}
\bibliography{biblio}

\end{document}